\documentclass[reqno]{amsart} 

\usepackage{graphicx, amsmath, amsthm, amssymb, amscd,color,hyperref,mathrsfs,tikz}
\usepackage{dynkin-diagrams}
\usepackage{bbold}

\hypersetup{
	colorlinks,
	citecolor=blue,
	filecolor=black,
	linkcolor=blue,
	urlcolor=black,
}

\newtheorem{theorem}{Theorem}[section]
\newtheorem{lemma}[theorem]{Lemma}
\newtheorem{proposition}[theorem]{Proposition}

\theoremstyle{definition}

\DeclareMathOperator{\Bip}{p_2}
\theoremstyle{remark}

\newcommand{\nc}{\newcommand}
\nc{\Z}{\mathbb Z}
\nc{\B}{\mathbb B}
\nc{\D}{\mathbb D}
\nc{\mb}{\mathbb}
\nc{\irr}{\mathrm{Irr}}

\numberwithin{equation}{section}

\title{Divisibility of Character Values of Representations of Coxeter Groups}
\author{Jyotirmoy Ganguly, Rohit Joshi}
\makeatletter\@addtoreset{chapter}{part}\makeatother%
\address{MURTI Research Center, GITAM (Deemed to be University),\\
	NH 207, Nagadenahalli, Doddaballapura Taluk, Bengaluru Rural,\\
	Karnataka - 562163, India.}
\email{jyotirmoy.math@gmail.com}  

\address{IISER Pune, Dr. Homi Bhabha Road,Pune-411008, Maharashtra, India} \email{rohitsj@students.iiserpune.ac.in} 

\begin{document}
	
\begin{abstract}
Let $d$ be a positive integer. We study the proportion of irreducible characters of infinite families of irreducible Coxeter groups whose values evaluated on a fixed element $g$ are divisible by $d$. For Coxeter groups of types $A_n, B_n$ and $D_n$, the proportion tends to $1$ as $n$ approaches infinity. For Dihedral groups, which are Coxeter groups of type $I_2(n)$, we compute the limit of the proportion. 
\end{abstract}		
	
	\maketitle
	\tableofcontents

\section{Introduction}	

%
%
%
%

Consider a chain of finite groups
\begin{center}
	\begin{tikzpicture}
		\node (A1) at (-.5,0) {$\mathscr{C}(G_n, G_0)$};
		\node (A2) at (.4,0) {$:$};
		\node (A3) at (1,0) {$G_0$};
		\node (A4) at  (2.3,0) {$G_1$};
        \node (A5) at  (3.6,0) {$G_2$};
        \node (A6) at  (4.9,0) {$\cdots$};		
        \node (A7) at  (6.2,0) {$G_n$};		
        \node (A8) at  (7.8,0) {$\cdots$};
		\draw[->](A3) to node [above]{$j_1$} (A4);
		\draw[->](A4) to node [above]{$j_2$} (A5);
		\draw[->](A5) to node [above]{$j_3$} (A6);
		\draw[->](A6) to node [above]{$j_n$} (A7);
		\draw[->](A7) to node [above]{$j_{n+1}$} (A8);
	\end{tikzpicture}
\end{center}
with injective group homomorphisms $j_k$, for $k\in \mathbb{N}$. We define $\iota_n=j_n\circ j_{n-1}\circ\cdots\circ j_1$. Write $\irr(G)$ for the set of irreducible representations of a finite group $G$. For $\pi$ a representation of $G$, let $\chi_{\pi}(g)$ denote the character value of $\pi$ evaluated on an element $g$ of $G$.
For $g\in G_0$ and $d\in \mathbb{N}$, we define
\begin{equation}
\mathscr{L}(\mathscr{C}(G_n, G_0),g,d)=\underset{n\rightarrow \infty}{\lim}\frac{\#\{\pi\in\irr(G_n)\mid \chi_{\pi}(\iota_n(g))\,\text{is divisible by}\,\, d\}}{|\irr(G_n)|},
\end{equation}
where we say $\chi_{\pi}(\iota_n(g))$ is divisible by $d$ whenever the algebraic number $\chi_{\pi}(\iota_n(g))/d$ is an (algebraic) integer.
The statistic $\mathscr{L}(\mathscr{C}(G_n, G_0),g,d)$ measures the proportion of character values of $G_n$ divisible by a fixed positive integer $d$ as $n$ approaches infinity. If $\mathscr{L}(\mathscr{C}(G_n, G_0),g,d)=1$, then we say that $100\%$ of the character values of $G_n$ are divisible by $d$ as $n\rightarrow\infty$.
The divisibility of character values for Symmetric groups were studied in \cite{gps}. Analogous study for the general linear groups over finite fields can be found in \cite{ss}.

In this article, we present a study of divisibility of character values of representations of irreducible Coxeter groups. Since our interest lies in the computation of the statistic $\mathscr{L}(\mathscr{C}(G_n, G_0),g,d)$, we consider only the infinite families of Coxeter groups. They are of types $A_n, B_n, D_n$ and $I_2(n)$ (see \cite{bourlie}).

It is well known that the character values of representations of Coxeter groups of types $A_n, B_n, D_n$ (Weyl Groups) are always integers [See \cite[Page 180]{hum}]. The divisibility results rely on a more general fact namely, for a finite group $G$ and $\pi \in \irr(G)$, $$\dfrac{\chi_\pi(g) \cdot [G : Z_G(g)]}{\dim \pi},$$
is an algebraic integer. Here $Z_G(g)$ is the centralizer of $g$ in $G$. For a reference see \cite[Exercise $6.9$]{ser}. For the case of the groups of type $I_2(n)$ we use some different techniques.


For a fixed positive integer $k$, consider the chain of Hyperoctahedral groups 
\begin{center}
	\begin{tikzpicture}
		\node (A1) at (-.5,0) {$\mathscr{C}(\mb {B}_{n}, \mb{B}_k)$};
		\node (A2) at (.4,0) {$:$};
		\node (A3) at (1,0) {$\mb B_k$};
		\node (A4) at  (2.5,0) {$\mb B_{k+1}$};
		\node (A5) at  (4.2,0) {$\mb B_{k+2}$};
		\node (A6) at  (6,0) {$\cdots$};		
		\node (A7) at  (7.2,0) {$\mb B_{n}$};		
		\node (A8) at  (8.8,0) {$\cdots$};
		\draw[->](A3) to node [above]{$j_{k+1}$} (A4);
		\draw[->](A4) to node [above]{$j_{k+2}$} (A5);
		\draw[->](A5) to node [above]{$j_{k+3}$} (A6);
		\draw[->](A6) to node [above]{$j_n$} (A7);
		\draw[->](A7) to node [above]{$j_{n+1}$} (A8);
	\end{tikzpicture}
\end{center}
Here $j_i : \mb B_{i-1} \to \mb B_{i}$ is the usual inclusion. (see Section \ref{conjB}.)

\begin{theorem}\label{mt}
	For any positive integers $d$ and $k$ and an element $g\in \mb{B}_k$, we have  
	$$\mathscr{L}(\mathscr{C}(\mb {B}_{n}, \mb{B}_k), g,d)=1.$$
\end{theorem}

A similar asymptotic holds when we restrict our attention to Demi-hyperoctahedral group $\mathbb{D}_n$ (see Theorem \ref{demi}).

Let $m=m_0, m_1, m_2,\ldots$ be integers such that $m_i$ divides $m_{i+1}$ for every $i$. We have chains of Dihedral groups

\begin{center}
	\begin{tikzpicture}
		\node (A1) at (-.5,0) {$\mathscr{C}(D_{m_n},D_m)$};
		\node (A2) at (.6,0) {$:$};
		\node (A3) at (1,0) {$D_m$};
		\node (A4) at  (2.3,0) {$D_{m_1}$};
		\node (A5) at  (3.6,0) {$D_{m_2}$};
		\node (A6) at  (4.9,0) {$\cdots$};		
		\node (A7) at  (6.2,0) {$D_{m_n}$};		
		\node (A8) at  (7.8,0) {$\cdots$};
		\draw[->](A3) to node [above]{$i_1$} (A4);
		\draw[->](A4) to node [above]{$i_2$} (A5);
		\draw[->](A5) to node [above]{$i_3$} (A6);
		\draw[->](A6) to node [above]{$i_n$} (A7);
		\draw[->](A7) to node [above]{$i_{n+1}$} (A8);
	\end{tikzpicture}
\end{center}

For a detailed discussion see Subsection \ref{di}. Let the group $D_m$ be generated by a rotation $r$ by angle $2\pi/m$ and a reflection $s$. We prove the following result:

\begin{theorem}\label{dihedral}
	We have	    $$\mathscr{L}(\mathscr{C}(D_{m_n}, D_m),r^l,2) = \dfrac{\gcd(m,4l)}{m}. $$
\end{theorem}

We prove a similar result for divisibility by $d$, where $d > 2$ (see Theorem \ref{d2}). Hence we cover all infinite families of irreducible Coxeter groups. 

Here is the layout of the paper. Section $2$ presents an overview of the results in the paper \cite{gps}. In Section $3$ we prove Theorem \ref{mt}. Section \ref{demihyper} contains a similar treatment for the Coxeter groups of type $D_n$. In the final Section \ref{dihedralgroups} we deal with the case of Dihedral groups (type $I_2(n)$).


{\bf Acknowledgements:} The authors would like to thank Steven Spallone for helpful conversations. The second author of this paper was supported by a post-doctoral fellowship
of National Board of Higher Mathematics India (NBHM).	

\section{Symmetric Groups}

The symmetric group $S_n$ is the Weyl group of type $A_n$. The divisibility of the character values of $S_n$ was studied in \cite{gps}. In this section we present a review of the paper.

Consider a chain of symmetric groups

\begin{center}
	\begin{tikzpicture}
		\node (A1) at (-.5,0) {$\mathscr{C}(S_{n},S_k)$};
		\node (A2) at (.6,0) {$:$};
		\node (A3) at (1,0) {$S_k$};
		\node (A4) at  (2.3,0) {$S_{k+1}$};
		\node (A5) at  (3.6,0) {$S_{k+2}$};
		\node (A6) at  (4.9,0) {$\cdots$};		
		\node (A7) at  (6.2,0) {$S_{n}$};		
		\node (A8) at  (7.8,0) {$\cdots$};
		\draw[->](A3) to node [above]{$j_1$} (A4);
		\draw[->](A4) to node [above]{$j_2$} (A5);
		\draw[->](A5) to node [above]{$j_3$} (A6);
		\draw[->](A6) to node [above]{$j_n$} (A7);
		\draw[->](A7) to node [above]{$j_{n+1}$} (A8);
	\end{tikzpicture}
\end{center}

For an element $g\in S_{n-1}$ we define $j_n(g)$ to be the element of $S_n$ which fixes $n$. So, if $g$ has
cycle type $\mu=(\mu_1,\dotsc,\mu_m)$, then the cycle type $j_n(g)$ is $(\mu_1,\dotsc,\mu_m,1)$.
The main result of the article provides the asymptotic nature of the proportion of the irreducible characters of $S_n$ divisible by a fixed positive integer.
\begin{theorem}\cite[Main Theorem]{gps}
	For any positive integers $k$ and $d$, 
	\begin{displaymath}
		\mathscr{L}(\mathscr{C}(S_n, S_k), g,d) = 1.
	\end{displaymath}
	
\end{theorem}
In particular, for any integer $d$, the probability that an irreducible character of $S_n$ has degree divisible by $d$ converges to $1$ as $n$ approaches infinity.

Let $\mathfrak{f}_\lambda$ denote the degree of irreducible representation of $S_n$ corresponding to partition $\lambda$.
In order to prove the main theorem, the authors focus on the divisibility properties of $\mathfrak{f}_\lambda$.
For each prime number $q$, let $v_q(m)$ denote the $q$-adic valuation of an integer $m$, in other words, $q^{v_q(m)}$ is the largest power of $q$ that divides $m$. Another key factor in their proof is the divisibility property of the degree of the irreducible representations of $S_n$.

\begin{theorem}\cite[Theorem $A$]{gps}\label{asymp1}
	$$\lim\limits_{n \to \infty}\dfrac{\#\{\lambda \vdash n \mid \upsilon_q(\mathfrak{f}_\lambda) \leq h + (q-1)\log n \}}{p(n)} = 0.$$
\end{theorem}

Here $h$ is a positive integer.
The proof of the above mentioned result is based on the theory of $q$-core towers.  This construction originated in the seminal paper \cite{macd} of Macdonald, and was developed further by Olsson in \cite{Olb}. The proof of the main theorem can be found in \cite[Section $3$]{gps}.

\section{Hyperoctahedral Groups}\label{hyper}

For this and the next section we follow the definitions and results from \cite{musili}. Consider the set $\mathfrak{X}_n=\{\pm1,\pm2,\ldots,\pm n\}$. Here we write $S_{2n}$ for the group of bijections from  $\mathfrak{X}_n$ to itself. For $n\geq 2$, we define the  $n$-th hyperoctahedral group $\mathbb{B}_n$ to be the following subgroup of $S_{2n}$:
$$\mathbb{B}_n=\{\sigma\in S_{2n}\mid \sigma(i)+\sigma(-i)=0, 1\leq i\leq n\}.$$
It is the Weyl group of types $B_n$ and $C_n$. An element in $\mathbb{B}_n$ which is 
\begin{enumerate}
	\item
	a product of two $l$-cycles of the form $(a_1, a_2, \ldots, a_l)(-a_1, -a_2,\ldots, -a_l)$ is called a positive $l$-cycle.
	\item 
	a $2l$-cycle of the form $(a_1, a_2,\ldots, a_l,-a_1,\ldots, -a_l)$ is called a negative $l$-cycle.
\end{enumerate}

For an  $l$-cycle $\sigma_j=(a_1,a_2,\ldots, a_l)$ we write $\bar{\sigma}_j=(-a_1, -a_2,\ldots, -a_l)$.

\subsection{Conjugacy Classes in $\mathbb{B}_n$} \label{conjB}

Suppose $n=a+b$, where $a$ and $b$ are two non-negative integers. Let $\alpha$ be a partition of $a$ and $\beta$ be a partition of $b$. Then the pair $(\alpha,\beta)$ is called a bipartition of $n$ and we write $(\alpha,\beta)\vDash n$ in this case.
The number of bipartitions  of $n$ is denoted by $p_2(n)$.

Any element $\sigma\in \mathbb{B}_n$ has a cycle decomposition $$\sigma=\sigma_1\bar{\sigma}_1\sigma_2\bar{\sigma_2}\cdots \sigma_r\bar{\sigma}_r\nu_1\nu_2\cdots\nu_s,$$ 
where $\sigma_i  \bar{\sigma}_i$ is a positive cycle and $\nu_k$ is a negative cycle. We take $|\sigma_i|=\lambda_i$ and $|\nu_j|=2\mu_j$. Then the pair of partitions $(\lambda,\mu)$ is called the cycle type of $\sigma$, where $\lambda=(\lambda_1,\lambda_2,\ldots,\lambda_r)$ and 
$\mu=(\mu_1,\ldots,\mu_s)$. In fact this gives a bijection between the set of conjugacy classes of $\mathbb{B}_n$ and the set of bipartitions of $n$. 
We denote the conjugacy class corresponding to partition $(\lambda, \mu)$ by $\mathfrak{C}_{\mb {B}_n}(\lambda, \mu)$. An element in $\mathfrak{C}_{\mb {B}_n}(\lambda, \mu)$ is denoted by $g(\lambda,\mu)$. For a fixed positive integer $k$, consider the chain of Hyperoctahedral groups 
\begin{center}
	\begin{tikzpicture}
		\node (A1) at (-.5,0) {$\mathscr{C}(\mb {B}_{n}, \mb{B}_k)$};
		\node (A2) at (.4,0) {$:$};
		\node (A3) at (1,0) {$\mb B_k$};
		\node (A4) at  (2.5,0) {$\mb B_{k+1}$};
		\node (A5) at  (4.2,0) {$\mb B_{k+2}$};
		\node (A6) at  (6,0) {$\cdots$};		
		\node (A7) at  (7.2,0) {$\mb B_{n}$};		
		\node (A8) at  (8.8,0) {$\cdots$};
		\draw[->](A3) to node [above]{$j_{k+1}$} (A4);
		\draw[->](A4) to node [above]{$j_{k+2}$} (A5);
		\draw[->](A5) to node [above]{$j_{k+3}$} (A6);
		\draw[->](A6) to node [above]{$j_n$} (A7);
		\draw[->](A7) to node [above]{$j_{n+1}$} (A8);
	\end{tikzpicture}
\end{center}

Consider an element $\sigma\in S_{2(n-1)}$. Take the map $j_n^*:S_{2(n-1)}\to S_{2n}$, such that $j_n^*(\sigma)\mid_{\mathfrak{X}_{n-1}}=\sigma$ and $j_n^*(\sigma)(\pm n)=\pm n$. 
Since $\mb{B}_k$ is a subgroup of $S_{2k}$, we define  $j_n :\mb B_{n-1} \to \mb B_n $ to be the restriction of $j_n^*$ to $\mb{B}_{n-1}$. It is easy to see that $$\iota_n(g(\lambda,\mu)) = g((\lambda, 1^{n-k}), \mu).$$



\subsection{Irreducible Representations of $\mathbb{B}_n$}

We write $\epsilon$ to denote the nontrivial character of $C_2$, the cyclic group of order $2$.
The $n$-th hyperoctahedral group can also be described as the wreath product $\mathbb{B}_n=C_2\wr S_n=C_2^n\rtimes S_n$. The normal subgroup $C_2^n\lhd \mathbb{B}_n$ has two $S_n$-invariant characters, namely the trivial one and $\eta=\epsilon\otimes\cdots\otimes \epsilon$.

   Let $\pi_{\lambda}$ denote the irreducible representation of $S_n$ corresponding to the partition $\lambda$. Consider two irreducible representations of $\mathbb{B}_n$ namely,
$$\pi_{\lambda}^0(x,\sigma)=\pi_{\lambda}(\sigma),\quad \pi_{\lambda}^1(x,\sigma)=\eta(x)\pi_{\lambda}(\sigma),$$
for $x\in C_2^n$ and $\sigma\in S_n$. Let $\pi_{\alpha,\beta}$ be defined as

\begin{equation}\label{repbn}
\pi_{\alpha,\beta}=\mathrm{Ind}_{\mathbb{B}_a\times \mathbb{B}_b}^{\mathbb{B}_n}\pi_{\alpha}^0\boxtimes\pi_{\beta}^1.
\end{equation}

In fact, the collection 
$$\{\pi_{\alpha,\beta}\mid \alpha\vdash a, \beta\vdash b, a+b=n\},$$
gives a complete set of representatives for the set of isomorphism classes of irreducible representations of $\mathbb{B}_n$. More details can be found about the irreducible representations of $\mathbb{B}_n$ in \cite{macdonald} and \cite{kinch}. 
Equation \eqref{repbn} suggests that 
\begin{equation}\label{dim}
\dim \pi_{\alpha,\beta}=\binom{n}{a}\mathfrak{f}_{\alpha}\mathfrak{f}_{\beta}.
\end{equation}

%
%
%
%
%
%

\subsection{Proof of Theorem \ref{mt}}

Let $Z_G(g)$ denote the centralizer of an element $g$ in $G$. To establish the divisibility of an irreducible character we use the following result in 
\cite[Exercise 6.9]{ser}. 

\begin{lemma}\label{sr}
Let $G$ be a finite group, $g\in G$ and $\pi\in \irr(G)$. Then	
$$\dfrac{\chi_{\pi}(g)}{\dim \pi}[G: Z_G(g)]$$
is an (algebraic) integer.
\end{lemma}

Since characters of $\mathbb{B}_n$ take integer values, Lemma \ref{sr} asserts that 
\begin{equation}\label{key}
	\dfrac{\chi_{\pi_{\alpha,\beta}}(w)}{\dim\pi_{\alpha,\beta}}[\mathbb{B}_n:Z_{\mathbb{B}_n}(w)]\in \mathbb{Z},
\end{equation}
for an element $w\in \mathbb{B}_n$.
Therefore 
\begin{equation}\label{hchar}
	\chi_{\pi_{\alpha,\beta}}(w)=m'\dim\pi_{\alpha, \beta}\dfrac{1}{[\mathbb{B}_n:Z_{\mathbb{B}_n}(w)]},
\end{equation}
where $m'\in \mathbb{Z}$. The next result gives an expression for the character values of representations of $\mb{B}_n$. We introduce one more notation here. The length of a partition $\nu$, denoted by $l(\nu)$, is equal to the number of parts of $\nu$.

\begin{lemma}\label{lemma1}
Let $\pi_{\alpha, \beta}\in \irr(\mb{B}_n)$ and $w\in \mathfrak{C}_{\mb {B}_n}((\lambda,1^{n-k}), \mu)$. Then 
$$	\chi_{\pi_{\alpha,\beta}}(w)=\frac{m}{2^k(n)_k}\binom{n}{|\alpha|}\mathfrak{f}_{\alpha}\mathfrak{f}_{\beta}2^{l(\lambda)+l(\mu)},$$
where $(n)_k=n(n-1)\cdots (n-k+1)$ and $m\in \mathbb{Z}$.
\end{lemma}

\begin{proof}
Consider an element $\sigma$ in $S_r$ with cycle type $\delta$. Let $m_i$ denote the number of $i$ cycles in $\sigma$. We write $z_{\delta}$ for the size of the centralizer of $\sigma$. Then
$$z_{\delta}=1^{m_1}m_1!\,2^{m_2}m_2!\ldots t^{m_t}m_t!.$$
From \cite[Page 313, Corollary 1]{tout} we have 
$$|\mathfrak{C}_{\mb {B}_n}((\lambda,1^{n-k}), \mu)|=[\mathbb{B}_n:Z_{\mathbb{B}_n}(w)]=\dfrac{2^nn!}{2^{l(\lambda)+l(\mu)}z_{\lambda}z_{\mu}}.$$

Following Equations \eqref{hchar} and \eqref{dim} one computes

\begin{align*}
	\chi_{\pi_{\alpha,\beta}}(w)&=\frac{m'}{2^nn!}2^{l(\lambda,1^{n-k})+l(\mu)}z_{(\lambda,1^{n-k})}z_{\mu}\dim\pi_{\alpha,\beta}\\
	&=\frac{m'}{2^nn!}\binom{n}{|\alpha|}\mathfrak{f}_{\alpha}\mathfrak{f}_{\beta}2^{l(\lambda)+l(\mu)+n-k}z_{(\lambda,1^{n-k})}z_{\mu}.\\
	&=\frac{m'}{2^kn!}\binom{n}{|\alpha|}\mathfrak{f}_{\alpha}\mathfrak{f}_{\beta}2^{l(\lambda)+l(\mu)}z_{(\lambda,1^{n-k})}z_{\mu}.
\end{align*}

Expanding the expression $z_{(\lambda,1^{n-k})}$ we obtain
\begin{align*}
	z_{(\lambda,1^{n-k})}&=(m_1+n-k)!2^{m_2}m_2!\ldots t^{m_t}m_t!	\\
	&=(n-k)!(n-k+1)(n-k+2)\cdots (m_1+n-k)2^{m_2}m_2!\ldots t^{m_t}m_t!.
\end{align*}
In short, we  write $z_{(\lambda,1^{n-k})}z_{\mu}=(n-k)!A$ where $A\in\mathbb{Z}$.

Therefore,
\begin{align*}
	\chi_{\pi_{\alpha,\beta}}(w)&=\frac{m'}{2^kn!}\binom{n}{|\alpha|}\mathfrak{f}_{\alpha}\mathfrak{f}_{\beta}2^{l(\lambda)+l(\mu)}(n-k)!A\\
	&=\frac{m}{2^k(n)_k}\binom{n}{|\alpha|}\mathfrak{f}_{\alpha}\mathfrak{f}_{\beta}2^{l(\lambda)+l(\mu)},
\end{align*}
where $m=m'A$.

\end{proof}


For each prime number $q$, let $v_q(m)$ denote the $q$-adic valuation of an integer $m$. We know that
 $\upsilon_q(n!)=\frac{n-\nu(n)}{q-1}$, where $\nu(n)$ is the sum of the coefficients in the $q$-nary expansion of $n$. Elementary calculation shows
\begin{equation}\label{vqnk}
\upsilon_q(2^k(n)_k)=\upsilon_q\left(\frac{2^k\cdot n!}{(n-k)!}\right)=\upsilon_q(2^k)+\frac{k+\nu(n-k)-\nu(n)}{q-1}\leq 2k+(q-1)\log_q(n).
\end{equation}

%


We now provide a proof of the main theorem. Let $\mathscr{P}(n)$ denote the set of partitions of $n$. We write $\lambda\vdash n$ to say that $\lambda$ is a partition of $n$.

\begin{proof}[Proof of \ref{mt}]
	
For this proof we essentially use the divisibility property of the degree of the irreducible representations of $S_n$.
Consider the set
$$\mathcal{T}(n)=\{\lambda\vdash n\mid \upsilon_q(\mathfrak{f}_{\lambda})\geq r+2k+(q-1)\log_qn\}.$$

Taking $h=r+2k$ in Theorem \ref{asymp1} gives 
$$\lim\limits_{n \to \infty}\dfrac{|\mathcal{T}(n)|}{p(n)}=1.$$
In other words, for a fixed $\delta>0$, there exists $N\in\mathbb{N}$ such that for $n>N$ we have
\begin{equation} \label{limit}
	\dfrac{|\mathcal{T}(n)|}{p(n)}>1-\delta.
\end{equation}

We construct one more subset of $\mathscr{P}(n)$ as follows

$$\mathcal{A}(n)=\{\lambda\vdash n\mid \upsilon_q(\mathfrak{f}_{\lambda})-\upsilon_q(2^k(n)_k)\geq r \}.$$
Equation \eqref{vqnk} clearly shows that $\mathcal{A}(n)\supseteq \mathcal{T}(n)$.	

Now we turn our attention to irreducible representations of $\mb{B}_n$. Using Lemma \ref{lemma1}, we obtain
\begin{equation}\label{prop}
\upsilon_q(\chi_{\pi_{\alpha,\beta}}(w))\geq \upsilon_q(\mathfrak{f}_{\alpha})+\upsilon_q(\mathfrak {f}_{\beta})-\upsilon_q(2^k(n)_k),
\end{equation}
where $w\in \mathfrak{C}_{\mb {B}_n}((\lambda,1^{n-k}), \mu)$. Fix a non-negative integer $r$.
We want to count the proportion of irreducible characters $\chi_{\pi_{\alpha,\beta}}$ for which $\upsilon_q(\chi_{\pi_{\alpha,\beta}}(w))\geq r$.
Towards that we define 
	\begin{equation}
		S=\{(\alpha,\beta)\vDash n\mid \upsilon_q(\mathfrak{f}_{\alpha})+\upsilon_q(\mathfrak{f}_{\beta})-\upsilon_q(2^k(n)_k)\geq r\}.
	\end{equation}
We aim to prove 
$$\lim\limits_{n \to \infty}\dfrac{|S|}{p_2(n)}=1.$$
To achieve this we define a subset of $S$ as follows:
	
	\begin{equation}\label{s'}
		S'=\bigsqcup_{a=\lfloor\frac{n}{2} \rfloor+1 }^{n}\left(\mathcal{A}(a)\times \{\beta\vdash n-a\}\right)\bigsqcup_{a=0 }^{\lfloor\frac{n}{2} \rfloor}\left(\{\alpha\vdash a\}\times \mathcal{A}(n-a)\right).
	\end{equation}

Using \eqref{limit} and \eqref{s'} one can provide a lower bound for $|S'|$.
Fix $\delta>0$. Then there exists $N\in \mb{Z}$ such that for $n>N$ we have

	\begin{align*}
		|S'|&= \sum_{a=\lfloor\frac{n}{2} \rfloor+1}^{n}|\mathcal{A}(a)|p(n-a)+\sum_{a=0}^{\lfloor\frac{n}{2} \rfloor}p(a)|\mathcal{A}(n-a)|\\
		&\geq \sum_{a=\lfloor\frac{n}{2} \rfloor+1}^{n}|\mathcal{T}(a)|p(n-a)+\sum_{a=0}^{\lfloor\frac{n}{2} \rfloor}p(a)|\mathcal{T}(n-a)|\\
		&>\sum_{a=\lfloor\frac{n}{2} \rfloor+1}^{n}(1-\delta)p(a)p(n-a)+\sum_{a=0}^{\lfloor\frac{n}{2} \rfloor}(1-\delta)p(a)p(n-a)\\
		&=(1-\delta)\sum_{a=0}^{n}p(a)p(n-a)\\
		&=(1-\delta)p_2(n).
	\end{align*}
	
Since $S'\subseteq S$, for $n$ large enough one obtains
	\begin{align}
		\dfrac{|S|}{p_2(n)}\geq \dfrac{|S'|}{p_2(n)}>1-\delta.
	\end{align}	
	
\end{proof}

\section{Demi-Hyperoctahedral Group}\label{demihyper}


The Demi-Hyperoctahedral group $\mathbb{D}_n$ is a subgroup of $\mathbb{B}_n$ defined as follows:
$$\mathbb{D}_n=\{\theta\in \mathbb{B}_n\mid \#\{i\mid \theta(i)<0, 1\leq i\leq n\}\,\,\text{is even}\}.$$
It is the Weyl group of type $D_n$. To study the asymptotic nature of the divisibility of the character values of $\mathbb{D}_n$, we follow a similar method. 
Fix a positive integer $k$. For the chain of groups 
\begin{center}
	\begin{tikzpicture}
		\node (A1) at (-.5,0) {$\mathscr{C}(\mb {D}_{n}, \mb{D}_k)$};
		\node (A2) at (.4,0) {$:$};
		\node (A3) at (1,0) {$\mb D_k$};
		\node (A4) at  (2.5,0) {$\mb D_{k+1}$};
		\node (A5) at  (4.2,0) {$\mb D_{k+2}$};
		\node (A6) at  (6,0) {$\cdots$};		
		\node (A7) at  (7.2,0) {$\mb D_{n}$};		
		\node (A8) at  (8.8,0) {$\cdots$};
		\draw[->](A3) to node [above]{$j_{k+1}$} (A4);
		\draw[->](A4) to node [above]{$j_{k+2}$} (A5);
		\draw[->](A5) to node [above]{$j_{k+3}$} (A6);
		\draw[->](A6) to node [above]{$j_n$} (A7);
		\draw[->](A7) to node [above]{$j_{n+1}$} (A8);
	\end{tikzpicture}
\end{center}

The map $j_{k}:\mb{D}_{k-1}\to \mb D_{k}$ is defined as the restriction of the map $j_k: \mb{B}_{k-1}\to \mb B_{k}$. Let $\theta$ be an element of $\mb{D}_k$ with cycle type $(\lambda,\mu)$.
Then $\iota_n(\theta)=g((\lambda,1^{n-k}),\mu)$. In fact, the conjugacy class of $g((\lambda,1^{n-k}),\mu)$ in $\mathbb{B}_n$ remains a single conjugacy class when restricted to $\mathbb{D}_n$. This is evident from the following result.

\begin{theorem}\label{conjDn}
	The conjugacy class $\mathfrak{C}_{\mathbb{B}_n}(\lambda,\mu)$ in $\mathbb{B}_n$ splits into a union of two conjugacy classes $\mathfrak{C}_{\mathbb{D}_n}^{+}(\lambda,\mu)\cup \mathfrak{C}_{\mathbb{D}_n}^{-}(\lambda,\mu)$ of $\mathbb{D}_n$ if and only if $\mu=0$ and all parts of $\lambda$ are even.
\end{theorem}



The irreducible representations of $\mathbb{D}_n$ are of two kinds:
\begin{itemize}
\item 
Let $(\alpha,\beta)\vDash n$ with $\alpha\neq \beta$. Then the irreducible representation $\pi_{\alpha,\beta}$ remains irreducible when restricted to $\mb D_n$. In this case, we write $\pi_{\alpha,\beta}^0$ to denote the restricted representation. Moreover, $\pi_{\alpha,\beta}^0$is isomorphic to $\pi_{\beta,\alpha}^0$. We denote the corresponding character value as $\varrho_{\alpha,\beta}=\chi^0_{\alpha,\beta}=\chi^0_{\beta,\alpha}$. 
\item 
Let $n>0$ be even and $\alpha$ be any partition of $n/2$. Then the irreducible representation $\pi_{\alpha,\alpha}$ of $\mathbb{B}_n$ when restricted to $\mathbb{D}_n$, is a sum of two non-isomorphic representations of $\mathbb{D}_n$, which are denoted by $\pi_{\alpha,\alpha}^{+}$ and $\pi_{\alpha,\alpha}^{-}$.

\end{itemize}

Therefore for $n>0$, we have
$$\irr(\mathbb{D}_n)=\{\pi^0_{\alpha,\beta}\mid \alpha\neq \beta\}\coprod \{\pi^{\pm}_{\alpha,\alpha}\mid \alpha\vdash n/2\}.$$
Therefore for $n$ odd, $|\irr(\mathbb{D}_n)|=\frac{1}{2}p_2(n)$. For $n>0$ and $n$ even, 
$$|\irr(\mathbb{D}_n)|=\frac{1}{2}(p_2(n)-p(n/2))+2p(n/2)=\frac{1}{2}(p_2(n)+3p(n/2)).$$

%

\begin{lemma}\label{repdn}
For $n$ even, we have 
$$\underset{n\rightarrow\infty}{\lim}\dfrac{\#\{\pi^{\pm}_{\alpha,\alpha}\mid \alpha\vdash n/2\}}{\irr(\mb{D}_n)}=0.$$
\end{lemma}

\begin{proof}
For $n$ even, we obtain an expression for $p_2(n)$ as follows:
\begin{align*}
	p_2(n)&=\sum_{r=0}^{n} p(r)p(n-r)\\
	&=\underset{0\leq r\leq n, r\neq n/2}{\sum}  p(r)p(n-r)+(p(n/2))^2.
\end{align*}
This gives 
\begin{equation}\label{eq1}
	p_2(n)>(p(n/2))^2.
\end{equation}
Using the above inequality we compute
$$\dfrac{\#\{\pi_{\alpha,\alpha}^{\pm}(\theta)\mid \alpha\vdash n/2\}}{|\irr(\mb{D}_n)|}=\dfrac{4p(n/2)}{\Bip(n)+3p(n/2)}<\dfrac{4p(n/2)}{(p(n/2))^2}=\dfrac{4}{p(n/2)}.$$
\end{proof}

\begin{theorem}\label{demi}
	For any integer $d$, we have
	$$\mathscr{L}(\mathscr{C}(\mathbb{D}_n, \mb{D}_k), g,d)=1.$$
\end{theorem}

\begin{proof}
	For $n$ odd, the set of irreducible characters are
	$$\{\varrho_{\alpha,\beta}\mid (\alpha,\beta)\vDash n, \alpha\neq \beta\}.$$
	Therefore, in this case
\begin{equation}
\#\{\pi\in\irr(\mathbb{D}_n)\mid \chi_{\pi}(\iota_n(g))\,\text{is divisible by}\,\, d\}=\frac{1}{2}\#\{\pi\in\irr(\mathbb{B}_n)\mid \chi_{\pi}(\iota_n(g))\,\text{is divisible by}\,\, d\}.
\end{equation}	
So the result follows from Theorem \ref{mt}.
	
	For $n$ even, the set of irreducible characters are 
	$$\{\varrho_{\alpha,\beta}\mid (\alpha,\beta)\vDash n, \alpha\neq \beta\}\sqcup \{\chi_{\alpha,\alpha}^{\pm}\mid \alpha\vdash n/2\}.$$

 Lemma \ref{repdn} shows that in the computation of $\mathscr{L}(\mathscr{C}(\mathbb{D}_n), g,d)$ we can ignore the subset $\{\pi^{\pm}_{\alpha,\alpha}\mid \alpha\vdash n/2\}$ of $\irr(\mb D_n)$. Towards computing the limit we have
	
	\begin{align*}
		\dfrac{\#\{\varphi\in\irr(\mathbb{D}_n)\mid \chi_{\varphi}(\theta)\,\,\text{is not divisible by $d$}\}}{|\irr(\mathbb{D}_n)|}&=\dfrac{\#\{(\alpha,\beta)\vDash n\mid \alpha\neq \beta, \varrho_{\alpha,\beta}(\theta)\,\,\text{is not divisible by $d$}\}}{\frac{1}{2}\left(p_2(n)+3p(n/2)\right)}\\
		& < \dfrac{2\#\{(\alpha,\beta)\vDash n\mid \alpha\neq \beta, \varrho_{\alpha,\beta}(\theta)\,\,\text{is not divisible by $d$}\}}{p_2(n)}\\
		&=\dfrac{\#\{(\alpha,\beta)\vDash n\mid \pi_{\alpha,\beta}\in \irr(\mb{B}_n),  \chi_{\alpha,\beta}(\theta)\,\,\text{is not divisible by $d$}\}}{p_2(n)}.
	\end{align*}

As $n$ approaches infinity, we obtain
\begin{align*}
1-\mathscr{L}(\mathscr{C}(\mathbb{D}_n, \mb{D}_k), \theta,d)<1-\mathscr{L}(\mathscr{C}(\mathbb{B}_n, \mb{B}_k), \theta,d).
\end{align*}
Theorem \ref{mt} asserts that $\mathscr{L}(\mathscr{C}(\mathbb{D}_n, \mb{D}_k), g,d)=1$.
\end{proof}

\section{Dihedral Groups}\label{dihedralgroups}

Finally we consider the family of dihedral Groups $D_m$ (type $I_2(m)$). We have 
$$D_m=\langle r,s\mid r^m=s^2=1,  srs=r^{-1}\rangle.$$
Let $\langle r \rangle$ denote the subgroup of $D_m$ generated by the element $r$. It gives the coset decomposition $D_m=\langle r \rangle \cup s \langle r \rangle$. Geometrically the group $D_m$ can be described as the group of isometries of the regular $m$-gon.

\subsection{Irreducible Representations of $D_m$}
The dimensions of the irreducible 
representations of $D_m$ are at most $2$.
For $m$ even, there are four linear characters
\begin{align*}
	\mathbb{1}: & \quad(r,s)\rightarrow (1,1),\\
	\chi_r: & \quad (r,s)\rightarrow (-1,1),\\
	\chi_s: & \quad (r,s)\rightarrow (1,-1),\\
	\chi_{rs}: & \quad (r,s)\rightarrow (-1,-1).
\end{align*}

The irreducible $2$ dimensional representations $\sigma_k:D_m\to \mathrm{GL}_2(\mathbb{R})$
are given by 
$$\sigma_k(r)=\begin{pmatrix}
	\cos\theta_k & -\sin\theta_k\\
	\sin\theta_k & \cos\theta_k
\end{pmatrix},  
\qquad
\sigma_k(s)=\begin{pmatrix}
	0 & 1 \\
	1 & 0
\end{pmatrix},
$$
where $\theta_k=\frac{2\pi k}{m}$ and $1\leq k\leq m/2-1$.  When $m$ is odd, $D_m$ has two linear characters $\mathbb{1}$ and $\chi_s$. The two dimensional irreducible representations are $\sigma_k$, where $1\leq k\leq (m-1)/2$. 

Consider the chain of Dihedral groups

\begin{center}
	\begin{tikzpicture}
		\node (A1) at (-.5,0) {$\mathscr{C}(D_{m_n},D_m)$};
		\node (A2) at (.6,0) {$:$};
		\node (A3) at (1,0) {$D_m$};
		\node (A4) at  (2.3,0) {$D_{m_1}$};
		\node (A5) at  (3.6,0) {$D_{m_2}$};
		\node (A6) at  (4.9,0) {$\cdots$};		
		\node (A7) at  (6.2,0) {$D_{m_n}$};		
		\node (A8) at  (7.8,0) {$\cdots$};
		\draw[->](A3) to node [above]{$i_1$} (A4);
		\draw[->](A4) to node [above]{$i_2$} (A5);
		\draw[->](A5) to node [above]{$i_3$} (A6);
		\draw[->](A6) to node [above]{$i_n$} (A7);
		\draw[->](A7) to node [above]{$i_{n+1}$} (A8);
	\end{tikzpicture}
\end{center}
where $m_i$ divides $m_{i+1}$. By the abuse of notation, we denote the generators of $D_{m_k}$ by $r$ and $s$ for all $k\in \mb{N}$. The inclusion maps $i_k: D_{m_{k-1}}\to D_{m_k}$ is given by $i_k(r^ls^j)=r^{cl}s^j$, where $c=\frac{m_k}{m_{k-1}}$, $0\leq l\leq m-1$ and $j\in \{0,1\}$. 
Geometrically one can visualize the chain as follows. Let $\Delta_k$ denote a regular $k$-gon. One can consider $\Delta_k$ as the convex polygon with vertices as the $k$-th roots of unity.
Then $D_{m_k}$ is the group of isometries of $\Delta_{m_k}$. We take the natural embedding of $\Delta_{m_{k-1}}$ inside $\Delta_{m_k}$. For $g\in D_{m_{k-1}}$, let $i_k(g)\in D_{m_{k}}$ be the extension of $g$.

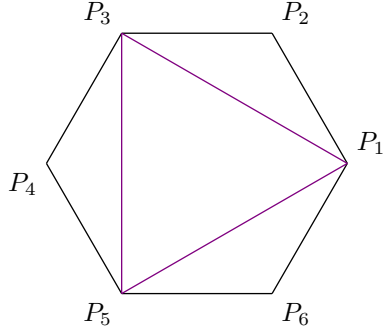
\begin{figure}
	\begin{center}
		\begin{tikzpicture}[scale=2, line width=.5pt]
			
			\foreach \i in {1,...,6} {
				\coordinate (H\i) at ({cos(60*\i)}, {sin(60*\i)});
			}
			
			\draw[black]      (H1) -- (H2);
			\draw[black]      (H2) -- (H3);
			\draw[black]     (H3) -- (H4);
			\draw[black]    (H4) -- (H5);
			\draw[black]       (H5) -- (H6);
			\draw[black]    (H6) -- (H1);
			\draw[violet]    (H6) -- (H2);
			\draw[violet]    (H6) -- (H4);
			\draw[violet]    (H2) -- (H4);
			
			\node at (H1) [above right] {$P_2$};
			\node at (H2) [above left] {$P_3$};
			\node at (H3) [below left] {$P_4$};
			\node at (H4) [below left] {$P_5$};
			\node at (H5) [below right] {$P_6$};
			\node at (H6) [above right] {$P_1$};
			
		\end{tikzpicture}
	\end{center}
	\caption{The figure shows an embedding of $\Delta_3=P_1P_3P_5$ inside $\Delta_6$.}
\end{figure}

\subsection{Divisibility of Character Values}

Let $\mathcal{O}$ denote the ring of integers in $\overline{\mathbb{Q}}$, the algebraic closure of $\mathbb{Q}$. Note that the character values of $D_m$ may not be rational integers. But being a sum of roots of unity they take values in the ring of integers $\mathcal{O}$. In this case, we say a positive integer $d$ divides the character value $\chi(g)$ if $\chi(g)/d\in \mathcal{O}$, where $g\in D_m$.

Take $\mathscr{S}_n=\{1\leq k< n/2\mid \gcd(k,n)=1\}$. From \cite{gurtas} we get that the minimal polynomial of 
$\cos\frac{2\pi}{n}$ is
\begin{equation}\label{minpol}
\widehat{\Psi}_n(x)=\prod_{k\in \mathscr{S}_n}(x-\cos(2\pi k/n)).
\end{equation}

The roots of $\widehat{\Psi}_n(x)$ are Galois conjugate to each other. Therefore $\cos(2\pi k/n)$ is Galois conjugate to $\cos(2\pi/n)$, whenever $\gcd(k,n)=1$.
Let $\zeta_n=e^{\frac{2\pi i}{n}}$ be the primitive root of unity. We write $\alpha_k(n)=\zeta_n^k+\zeta_n^{-k}$. Observe that $\chi_{\sigma_k}(r)=\alpha_k(m)$, where $r\in D_m$. Next we present a proof of an elementary result for the convenience of the reader.

\begin{lemma}\label{lemmaai}
We have $\alpha_1(n)\in \mb{Q}$ if and only if $n\in \{1,2,3,4,6\}$.
\end{lemma}

\begin{proof} 	
Note that $\alpha_1(n)$ is an algebraic integer. Therefore $\alpha_1(n)\in \mb{Q}$ implies $\alpha_1(n)\in \mb{Z}$. Moreover, $\alpha_1=2\cos \frac{2\pi}{n}$. Thus $\alpha_1\in \mb{Z}$ if and only if $\cos \frac{2\pi}{n}$ is an integral multiple of $1/2$. This is true if and only if $n\in \{1,2,3,4,6\}$.
\end{proof}

\begin{lemma}\label{lemmaai2}
For a positive integer $d\geq2$, if $\alpha_1(n)\notin \mb{Q}$ then $\alpha_1(n)/d\notin \mathcal{O}$.
\end{lemma}

\begin{proof}
We start with the assumption that $\alpha_1(n)\notin \mb{Q}$.	
From \cite[Proposition $2.16$]{wash} we know that $\mathbb{Z}[\alpha_1(n)]$ is the ring of integers for $\mathbb{Q}[\alpha_1(n)]$. Let $e$ be the degree of the minimal polynomial of $\alpha_1(n)$. Observe that $\alpha_1(n)\in \mathcal{O}$. Therefore, if $j>e$ by an inductive process one has  $\alpha_1(n)^j=\sum_{i=1}^ec_i(\alpha_1(n))^i$, where $c_i\in \mb{Z}$.
 If $\alpha_1(n)/d\in \mathcal{O}$ then 
$$\alpha_1/d=\sum_{i=1}^b a_i(\alpha_1(n))^i,$$
where $a_i\in \mathbb{Z}$ and and $b\leq e$. So $\alpha_1$ satisfies the polynomial $f(x)=\sum_{i=1}^b a_ix^i-x/d$. Therefore we must have $b=e$. But $f(x)$ can't be transformed into a monic and integral polynomial simultaneously by multiplying it with a rational number.
Therefore $\alpha_1(n)/d$ is not an algebraic integer.
\end{proof}

For a positive integer $d\geq 2$, we have 
$$\alpha_k(n)/d=\frac{2}{d}\cos\left(\frac{2\pi k}{n}\right)=\frac{2}{d}\cos\left(\frac{2\pi k'}{n'}\right),$$	
where $k'=k/\gcd(n,k)$ and $n'=n/\gcd(n,k)$. Since $\gcd(k',n')=1$, $\alpha_k(n)/d$ is Galois conjugate to $\alpha_1(n')/d$. Moreover, we use the fact that Galois conjugate of an algebraic integer is an algebraic integer. In particular, we have $\alpha_k(n)/d\in \mathcal{O}$ if and only if $\alpha_1(n')/d\in \mathcal{O}$.

\begin{proposition} \label{propai1}
	Let $k$ and $n$ be positive integers. Then TFAE:
	\begin{enumerate}
	\item 
	$\alpha_k(n)$ is divisible by $2$
	\item 
	$k/n$ is an integral multiple of $1/4$.
	\item 
	$\alpha_k(n)$ takes values from the set $\{2,0,-2\}$.
	\end{enumerate}
\end{proposition}

\begin{proof}

Observe that  $\alpha_k(n)/2\in \mathcal{O}$ if and only if $\alpha_1(n')/2\in \mathcal{O}$. If $\alpha_1(n')/2\in \mathcal{O}$,  Lemma \ref{lemmaai2}  implies $\alpha_1(n')\in \mb{Q}$. Further, Lemma \ref{lemmaai} shows that the possible values of $n'$ are $\{1,2,3,4,6\}$. Evaluating $\alpha_1(n')/2$ for these values one obtains $\alpha_1(n')/2\in \mathcal{O}$ if and only if $n'\in \{1,2,4\}$. An equivalent condition is that $k$ is an integral multiple of $n/4$. Elementary calculation shows that $\alpha_k(n)$ takes the values $\{2,0,-2\}$ when $k$ is an integral multiple of $n/4$.

\end{proof}

\begin{proposition}\label{propai2}
Let $k, d$ and $n$ be positive integers with $d>2$. Then TFAE:
\begin{enumerate}
\item 
$\alpha_k(n)$ is divisible by $d$.
\item 
$k/n$ is an odd multiple of $1/4$
\item 
$\alpha_k(n)=0$.
\end{enumerate}
\end{proposition}

\begin{proof}
For $d>2$, similar argument as in the previous proposition shows that $\alpha_k(n)/d\in \mathcal{O}$ if and only if $n'=4$. Equivalently, one has the condition that $k$ is an odd multiple of $n/4$. 
\end{proof}

\subsection{Results on Dihedral Groups}\label{di}
We now prove the results related to Dihedral groups.




\begin{proof}[Proof of Theorem \ref{dihedral}]	
First we consider the case when $m$ is even.
In that case, $|\irr(D_{m})|=m/2+3$. For a two dimensional representation $\sigma_k\in \irr({D}_m)$, we have $\chi_{\sigma_k}(r^l)=2\cos (2\pi lk/m)$. The required ratio $\cos (2\pi lk/m)\in \mathcal{O}$ when $lk/m$ is an integral multiple of $1/4$ (see Proposition \ref{propai1}). So we need $k=v\frac{m}{4l}$ where $v\in \mb{Z}$. We have $m$ divides $4lk$ if and only if $\frac{m}{\gcd(m,4l)}$ divides $\frac{4lk}{\gcd(m,4l)}$. Since $\frac{m}{\gcd(m,4l)}$ and $\frac{4l}{\gcd(m,4l)}$
are coprime, we have $m$ divides $4lk$ if and only if $\frac{m}{\gcd(m,4l)}$ divides $k$.


For $D_m$ one has
\begin{align*}
\#\{k\mid \chi_{\sigma_k}(r^l)\,\text{is divisible by}\,\, 2\} &= |\{k\mid \dfrac{m}{\gcd(m,4l)}\,\,\text{divides}\,\, k, 1\leq k\leq m/2-1\}|\\
&=\left\lfloor  \frac{m/2-1}{m/\gcd(m,4l)} \right\rfloor\\
&= \left\lfloor \frac{(m-2)\gcd(m,4l)}{2m}\right\rfloor.
\end{align*}

So we compute
\begin{align*}
\#\{k\mid \chi_{\sigma_k}(\iota_n(r^l))\,\text{is divisible by}\,\, 2\}&= |\{k\mid \dfrac{m_n}{\gcd(m_n,\frac{4m_nl}{m})}\mid k, 1\leq k\leq \frac{m_n}{2}-1\}|\\
&=|\{k\mid \dfrac{m_n}{\frac{m_n}{m}\gcd(m,4l)}\mid k, 1\leq k\leq \frac{m_n}{2}-1\}|\\ 
&=\left\lfloor \frac{(m_n-2)\gcd(m,4l )}{2m}\right\rfloor.
\end{align*}

We have
\begin{equation}\label{asymdi1}
\frac{\#\{\pi\in\irr(D_{m_n})\mid \chi_{\pi}(\iota_n(r^{l}))\,\text{is divisible by}\,\, 2\}}{|\irr(D_{m_n})|}= \dfrac{1}{(m_n)/2+3}\left\lfloor \frac{(m_n-2)\gcd(m,4l)}{2m}\right\rfloor
\end{equation}
From the property of the floor function we get the inequality
\begin{equation*}
\frac{(m_n-2)\gcd(m,4l)}{2m((m_n)/2+3)}-1<\dfrac{1}{(m_n)/2+3}\left\lfloor \frac{(m_n-2)\gcd(m,4l)}{2m}\right\rfloor\leq\frac{(m_n-2)\gcd(m,4l )}{2m((m_n)/2+3)}.
\end{equation*}

One computes

$$\underset{n\rightarrow \infty}{\lim}\frac{(m_n-2)\gcd(m,4l )}{2m((m_n)/2+3)}=\underset{n\rightarrow \infty}{\lim}\left\{\frac{\gcd(m,4l)}{m+\frac{6m}{m_n}}-\frac{2\gcd(m,4l)}{m(m_n+6)}\right\}=\dfrac{\gcd(m,4l)}{m}.$$

Taking limit of $n$ to infinity we obtain the desired result. The proof for the case when $m$ is odd follows similarly.
\end{proof}

\begin{theorem}\label{ratio1}
We have	
$$\mathscr{L}(\mathscr{C}(D_{m_n}, D_m),g,2)=1$$
if and only if $g$ is a reflection or $g\in H$, where
\[
H=
\begin{cases}
\{e\}, & \text{$m$ is odd},\\
Z(D_{m}), & m\equiv 2\pmod 4,\\
\langle r^{m/4}\rangle, & m\equiv 0\pmod 4.
\end{cases}
\]
\end{theorem}

\begin{proof}
If $g$ is a reflection, then $g\in s\langle r \rangle$. In that case $\chi_{\pi}(\iota_n(g))=0$ unless $\dim\pi=1$. For the one dimensional representations $\pi$, $\chi_{\pi}(\iota_n(g))/2\notin \mathcal{O}$. Note that the number of one dimensional representations is at most $4$ for any $m_n$. Therefore they don't affect the asymptotic behavior of the ratio.

If $g\in \langle r\rangle$, then $\frac{\gcd(m,4l)}{m}=1$ if and only if $m$ divides $4l$. This gives $4l\in \{m, 2m, 3m, 4m\}$ as $1\leq l\leq m$. The theorem follows from this.
\end{proof}

Next, we study the ratio of the character values divisible by $d$, where $d>2$.

\begin{theorem}
For $d>2$,
\[
\mathscr{L}(\mathscr{C}(D_{m_n}, D_m),r^l,d)=\begin{cases}
\dfrac{\mathrm{gcd}(m, 4l)}{2m},& \text{if  $\upsilon_2(m)\geq 2,\,\, \upsilon_2(l)\leq \upsilon_2(m)-2$},\\
0, & \text{otherwise}.
\end{cases}
\]

\end{theorem}

\begin{proof}
From Proposition \ref{propai2} we have $\chi_{\sigma_k}(r^l)=\frac{2}{d}\cos(\frac{2\pi lk}{m})\in \mathcal{O}$ if and only if $kl/m$ is an odd multiple of $1/4$. Equivalently, $k=(2u+1)\frac{m}{4l}$, where $u\in \mathbb{Z}$. Note that an integer integral multiple of $m/4l$ is an integral multiple of $m/\gcd(m,4l)$. An odd multiple of $m/4l$ is an integer, if and only if $\frac{m}{\gcd (m,4l)}$ is an odd multiple of $m/4l$. An equivalent condition is $4l/\gcd(m,4l)$ is odd. 

Observe that if $\upsilon_2(m)<2$, then $4l/\gcd(m,4l)$ is even. Therefore $\frac{2}{d}\cos(\frac{2\pi lk}{m})\notin \mathcal{O}$. Therefore we consider $m\in \mb{Z}$ such that $\upsilon_2(m)\geq 2$. One easily computes $$\upsilon_2\left(4l/\gcd(m,4l)\right)=\upsilon_2\left(l/\gcd(m/4,l)\right)=\upsilon_2(l)-\min\{\upsilon_2(m)-2,\upsilon_2(l)\}.$$
Therefore, the integer $4l/\gcd(m,4l)$ is odd if and only if $\upsilon_2(l)\leq \upsilon_2(m)-2$.

We have
\begin{equation*}	
\begin{split}
\#\{k\mid \chi_{\sigma_k}(r^l)\,\text{is divisible by}\,\, d\}&=|\{k\mid \text{$k$ is an odd multiple of  }\,\dfrac{m}{\gcd(m,4l)}, 1\leq k\leq m/2-1\}| \\
&= \left\lfloor \frac{1}{2}\left\{\dfrac{(m-2)/2}{m/\gcd(m,4l)}+1\right\}\right\rfloor \\
&= \left\lfloor \dfrac{(m-2)\gcd(m,4l)+2m}{4m}\right\rfloor.
\end{split}
\end{equation*}
From the property of the floor function, one has

$$\dfrac{(m-2)\gcd(m,4l)+2m}{4m}-1<\left\lfloor \dfrac{(m-2)\gcd(m,4l)+2m}{4m}\right\rfloor\leq \dfrac{(m-2)\gcd(m,4l)+2m}{4m}.$$

Simplyfying the expression we have 
$$\dfrac{(m-2)\gcd(m,4l)}{4m}-1/2<\left\lfloor \dfrac{(m-2)\gcd(m,4l)+2m}{4m}\right\rfloor\leq \dfrac{(m-2)\gcd(m,4l)}{4m}+1/2.$$

For $D_{m_n}$ we obtain
\begin{equation}
\dfrac{(m_n-2)\gcd(m_n,4\frac{m_n}{m}l)}{4m_n}-1/2<\#\{k\mid \chi_{\sigma_k}(\iota_n(r^l))\,\text{is divisible by}\,\, d\}\leq \dfrac{(m_n-2)\gcd(m_n,4\frac{m_n}{m}l)}{4m_n}+1/2.
\end{equation}

This yields
\begin{align*}
\dfrac{(m_n-2)\gcd(m,4l)-2m}{4m(m_n/2+3)}<\frac{\#\{k\mid \chi_{\sigma_k}(\iota_n(r^l))\,\text{is divisible by}\,\, d\}}{|\irr(D_{m_n})|}\leq \dfrac{(m_n-2)\gcd(m,4l)+2m}{4m(m_n/2+3)}.
\end{align*}
The desired result emerges as $n$ tends to infinity.

\end{proof}

\begin{theorem}\label{d2}
For $d>2$,
\[\mathscr{L}(\mathscr{C}(D_{m_n}, D_m),g,d)=
\begin{cases}
	\frac{1}{2},& \text{if \, $\upsilon_2(m)\geq 2, \,g\in \{r^{m/4}, r^{3m/4}\}$},\\
	1, & \text{if \, $g$ is a reflection}.
\end{cases}
\]
\end{theorem}

\begin{proof}
The proof is similar to that of Theorem \ref{ratio1}.
\end{proof}

\bibliographystyle{alpha}
\bibliography{mybib}	

\newcommand{\etalchar}[1]{$^{#1}$}
\begin{thebibliography}{Hum92}

\bibitem[Bou94]{bourlie}
Nicolas Bourbaki.
\newblock Lie groups and lie algebras.
\newblock {\em Elements of the History of Mathematics}, pages 247--267, 1994.

\bibitem[GK78]{kinch}
Ladnor Geissinger and Dennis Kinch.
\newblock Representations of the hyperoctahedral groups.
\newblock {\em Journal of algebra}, 53(1):1--20, 1978.

\bibitem[GPS19]{gps}
Jyotirmoy Ganguly, Amritanshu Prasad, and Steven Spallone.
\newblock On the divisibility of character values of the symmetric group.
\newblock {\em electronic journal of cmbinatorics}, 2019.

\bibitem[Gur16]{gurtas}
Yusuf~Z Gurtas.
\newblock The minimal polynomial of cos (2$\pi$/n).
\newblock {\em Communications of the Korean Mathematical Society},
  31(4):667--682, 2016.

\bibitem[Hum92]{hum}
James~E Humphreys.
\newblock {\em Reflection groups and Coxeter groups}.
\newblock Number~29. Cambridge university press, 1992.

\bibitem[Mac71]{macd}
Ian~G Macdonald.
\newblock On the degrees of the irreducible representations of symmetric
  groups.
\newblock {\em Bulletin of the London Mathematical Society}, 3(2):189--192,
  1971.

\bibitem[Mac98]{macdonald}
Ian~Grant Macdonald.
\newblock {\em Symmetric functions and Hall polynomials}.
\newblock Oxford university press, 1998.

\bibitem[MM11]{musili}
C~Musili and C~Musili.
\newblock {\em Representations of finite groups}.
\newblock Hindustan Book Agency, 2011.

\bibitem[Ols93]{Olb}
J\o rn~B. Olsson.
\newblock {\em Combinatorics and representations of finite groups}, volume~20
  of {\em Vorlesungen aus dem Fachbereich Mathematik der Universit\"at GH Essen
  [Lecture Notes in Mathematics at the University of Essen]}.
\newblock Universit\"at Essen, Fachbereich Mathematik, Essen, 1993.

\bibitem[S{\etalchar{+}}77]{ser}
Jean-Pierre Serre et~al.
\newblock {\em Linear representations of finite groups}, volume~42.
\newblock Springer, 1977.

\bibitem[SS24]{ss}
V.~Shah and S.~Spallone.
\newblock The divisibility of gl(n,q) character values.
\newblock {\em https://arxiv.org/pdf/2408.14046v2}, 2024.

\bibitem[Tou21]{tout}
O~Tout.
\newblock The center of the wreath product of symmetric group algebras.
\newblock {\em Algebra and Discrete Mathematics}, 31(2), 2021.

\bibitem[Was12]{wash}
Lawrence~C Washington.
\newblock {\em Introduction to cyclotomic fields}, volume~83.
\newblock Springer Science \& Business Media, 2012.

\end{thebibliography}
	
\end{document}